\documentclass[11pt,a4paper]{amsart}
\usepackage{graphicx,multirow,array,amsmath,amssymb}

\newtheorem{theorem}{Theorem}

\newtheorem{proposition}[theorem]{Proposition}
\newtheorem{claim}{Claim}

\newcommand{\R}{\mathbb R}
\newcommand{\adj}{\thicksim}
\newcommand{\nadj}{\nsim}

\begin{document}

\title[More intrinsically knotted graphs with 22 edges]
{More intrinsically knotted graphs with 22 edges and the restoring method}
\author[H. Kim]{Hyoungjun Kim}
\address{Institute of Mathematical Sciences, Ewha Womans University, Seoul 03760, Korea}
\email{kimhjun@korea.ac.kr}
\author[T. Mattman]{Thomas Mattman}
\address{Department of Mathematics and Statistics, California State University, Chico, Chico CA 95929-0525, USA}
\email{TMattman@CSUChico.edu}
\author[S. Oh]{Seungsang Oh}
\address{Department of Mathematics, Korea University, Seoul 02841, Korea}
\email{seungsang@korea.ac.kr}

\thanks{2010 Mathematics Subject Classification: 57M25, 57M27, 05C10}
\thanks{The corresponding author(Seungsang Oh) was supported by the National Research Foundation of Korea (NRF) grant funded by the Korea government(MSIP) (No. NRF-2014R1A2A1A11050999).}
\thanks{The first author was supported by Basic Science Research Program through the National Research Foundation of Korea (NRF) funded by the Ministry of Education(2009-0093827).}

\begin{abstract}
A graph is called intrinsically knotted if every embedding of the graph contains a knotted cycle.
Johnson, Kidwell and Michael, and, independently, Mattman showed that intrinsically knotted graphs have 
at least 21 edges.
Recently Lee, Kim, Lee and Oh, and, independently, Barsotti and Mattman, 
showed that $K_7$ and the 13 graphs obtained from $K_7$ by $\nabla Y$ moves are 
the only intrinsically knotted graphs with 21 edges.
Also Kim, Lee, Lee, Mattman and Oh showed that there are exactly three triangle-free intrinsically knotted 
graphs with 22 edges having at least two vertices of degree 5.
Furthermore, there is no triangle-free intrinsically knotted graph with 22 edges 
that has a vertex with degree larger than 5.
In this paper we show that there are exactly five triangle-free intrinsically knotted graphs with 22 edges 
having exactly one degree 5 vertex.
These are Cousin 29 of the $K_{3,3,1,1}$ family, Cousins 97 and 99 of the $E_9+e$ family and 
two others that were previously unknown.
\end{abstract}

\maketitle

\section{introduction} \label{sec:int}

Throughout the paper we will take an embedded graph to mean a graph embedded in $\R^3$.
We call a graph $G$ {\em intrinsically knotted\/} 
if every embedding of the graph contains a non-trivially knotted cycle.
Conway and Gordon~\cite{CG} showed that $K_7$, the complete graph on seven vertices, 
is an intrinsically knotted graph.
Foisy~\cite{F} showed that $K_{3,3,1,1}$ is also intrinsically knotted.
A graph $H$ is a {\em minor\/} of another graph $G$ 
if it can be obtained from a subgraph of $G$ by contracting some edges.
If a graph $G$ is intrinsically knotted and has no proper minor that is intrinsically knotted, 
$G$ is said to be {\em minor minimal intrinsically knotted\/}.
Robertson and Seymour's~\cite{RS} Graph Minor Theorem implies 
that there are only finitely many minor minimal intrinsically knotted graphs, 
but finding the complete set is still an open problem.
A $\nabla Y$ {\em move\/} is an exchange operation on a graph 
that removes all edges of a 3-cycle $abc$ and 
then adds a new vertex $v$ and three new edges $va, vb$ and $vc$.
Its reverse operation is called a $Y \nabla$ {\em move\/} as in Figure~\ref{fig:Y}.

\begin{figure}[h]
\includegraphics{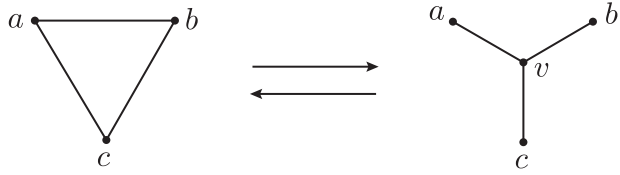}
\caption{$\nabla Y$ and $Y \nabla$ moves}
\label{fig:Y}
\end{figure}

The $\nabla Y$ move preserves intrinsic knottedness~\cite{MRS}. 
Sachs~\cite{S} observed this for intrinsic linking and the proof for intrinsic knotting is similar.
With this in mind, we concentrate on triangle-free graphs in this paper.
It is known \cite{CG, F, KS} that $K_7$ and the 13 graphs obtained from $K_7$ by $\nabla Y$ moves, 
and $K_{3,3,1,1}$ and the 25 graphs obtained from $K_{3,3,1,1}$ by $\nabla Y$ moves are 
minor minimal intrinsically knotted.

Johnson, Kidwell and Michael~\cite{JKM}, and, independently, Mattman~\cite{M} showed 
that intrinsically knotted graphs have at least 21 edges.
Recently two groups, working independently, showed 
that $K_7$ and the 13 graphs obtained from $K_7$ by $\nabla Y$ moves are 
the only minor minimal intrinsically knotted graphs with 21 edges.
Lee, Kim, Lee and Oh~\cite{LKLO} proceeded by showing 
that the only triangle-free intrinsically knotted graphs with 21 edges are $H_{12}$ and $C_{14}$.
Barsotti and Mattman~\cite{BM} relied on connections with 2-apex graphs.

We say that a graph is {\em intrinsically knotted or completely 3-linked\/} 
if every embedding of the graph contains a nontrivial knot or a 3-component link 
each of whose 2-component sublinks is non-splittable.
Hanaki, Nikkuni, Taniyama and Yamazaki~\cite{HNTY} constructed 20 graphs 
derived from $H_{12}$ and $C_{14}$ by $Y \nabla$ moves, 
and they showed that these graphs are minor minimal intrinsically knotted or completely 3-linked graphs.
Furthermore they showed that the six graphs $N_9$, $N_{10}$, $N_{11}$, $N'_{10}$, $N'_{11}$ and $N'_{12}$ 
are not intrinsically knotted.
This was also shown by Goldberg, Mattman and Naimi~\cite{GMN}, independently.

We say two graphs $G$ and $G'$ are {\em cousins\/} 
if $G'$ is obtained from $G$ by a finite sequence of $\nabla Y$ and $Y \nabla$ moves.
The set of all cousins of $G$ is called the $G$ {\em family\/}.
The $K_{3,3,1,1}$ family consists of 58 graphs. 
As above, 26 of these graphs were previously known to be minor minimal intrinsically knotted \cite{F, KS}.
Goldberg et al.~\cite{GMN} showed that the remaining 32 graphs are also minor minimal intrinsically knotted. 
They also studied the $E_9+e$ family which has 110 graphs.
As in Figure~\ref{fig:e9}, the graph $E_9+e$ is obtained from $N_9$ by adding the new edge $e$.
They showed that all of these graphs are intrinsically knotted, 
and exactly 33 of them are minor minimal intrinsically knotted.

\begin{figure}[h]
\includegraphics{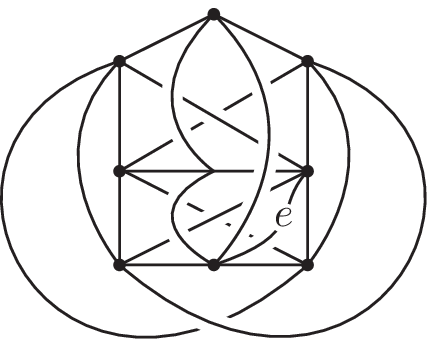}
\caption{$E_9+e$}
\label{fig:e9}
\end{figure}

By combining the $K_{3,3,1,1}$ family and the $E_9+e$ family, we have 168 graphs of size 22 
that were already known to be intrinsically knotted. 
Among these, 14 are triangle-free, namely, the four cousins 29, 31, 42 and 53 in the $K_{3,3,1,1}$ family 
and the ten cousins 43, 56, 87, 89, 94, 97, 99, 105, 109 and 110 in the $E_9+e$ family.
In earlier work~\cite{KLLMO} we showed that  
a triangle-free intrinsically knotted graph with 22 edges has no vertices of degree 6 or more 
and classified those that have at least two vertices of degree 5.
Cousins 94 and 110 of the $E_9+e$ family are two of the three examples.

In this paper we determine graphs which have exactly one vertex of degree 5. 
Three examples, shown in Figure~\ref{fig:3eg}, are Cousin 29 of the $K_{3,3,1,1}$ family 
and Cousins 97 and 99 of the $E_9+e$ family.
While Cousin 29 is minor minimal intrinsically knotted, 
Cousins 97 and 99 have $H_{11}$ in the $K_7$ family as a minor, 
as may be seen by contracting the edges $e_1$ and $e_2$, respectively.
The remaining nine examples in these two families have maximal degree 4.
 
\begin{figure}[h]
\includegraphics{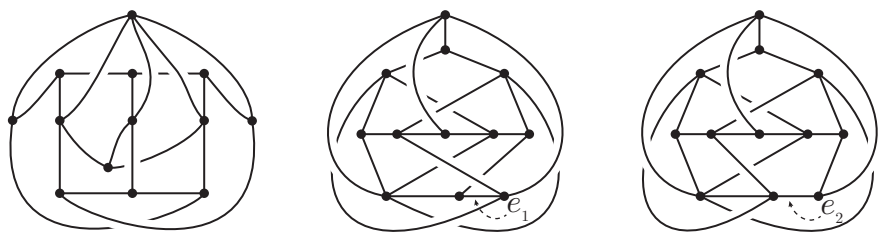}
\caption{Cousin 29 of the $K_{3,3,1,1}$ family and Cousins 97 and 99 of the $E_9+e$ family}
\label{fig:3eg}
\end{figure}

We introduce two previously unknown graphs $U_{12}$ and $U'_{12}$, shown in Figure~\ref{fig:new}, 
which are triangle-free intrinsically knotted (but not minor minimal) graphs 
with 22 edges having a unique vertex of degree 5.
Note that these graphs are not members of the $K_{3,3,1,1}$ family or the $E_9+e$ family.
Indeed, if we contract the edge $e_1$ of $U_{12}$ (similarly $e_2$ of $U'_{12}$),
we get $H_{11}$ in the $K_7$ family.

\begin{figure}[h]
\includegraphics{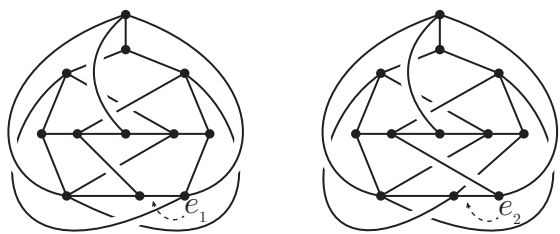}
\caption{New intrinsically knotted graphs $U_{12}$ and $U'_{12}$}
\label{fig:new}
\end{figure}

Our overall goal is to determine the set of intrinsically knotted graphs with 22 edges.
Kim, Lee, Lee, Mattman and Oh found all triangle-free intrinsically knotted graphs with 22 edges 
that have at least two vertices of degree 5 
and showed there are none with maximum degree larger than 5~\cite{KLLMO}.
In this paper we find all triangle-free intrinsically knotted graphs with 22 edges 
that have exactly one degree 5 vertex.

\begin{theorem} \label{thm:main}
There are exactly five triangle-free intrinsically knotted graphs 
with 22 edges having a unique degree 5 vertex. 
They are Cousin 29 of the $K_{3,3,1,1}$ family, Cousins 97 and 99 of the $E_9+e$ family 
and two previously unknown graphs, $U_{12}$ and $U'_{12}$, shown in Figure~\ref{fig:new}.
\end{theorem}

To complete the classification, it remains to investigate graphs with maximal degree 4. 
It's easy to see that there must be at least two vertices of that degree, 
as in the nine remaining examples in the $K_{3,3,1,1}$ and $E_9+e$ families.

In Section~\ref{sec:ter}, we introduce terminology and give an overview of the structure of our proof.
Sections~\ref{sec:pro1} and \ref{sec:pro2} are devoted to the proof of the main theorem.
Especially, in Section~\ref{sec:pro2} we introduce the restoring method used frequently in crucial parts.

\section{terminology} \label{sec:ter}

The notation and terminology used in this paper follow those employed 
in the previous paper~\cite{KLLMO}.
Let $G = (V,E)$ denote a triangle-free graph with 22 edges 
where $V$ and $E$ are the sets of vertices and edges, respectively.
For distinct vertices $a$ and $b$, 
let $G \backslash \{ a,b \}$ denote the graph obtained from $G$ by deleting two vertices $a$ and $b$.
Deleting a vertex means removing the vertex, interiors of all edges adjacent to the vertex 
and remaining isolated vertices.
Let $G_{a,b}$ denote the graph obtained from $G \backslash \{ a,b \}$ by deleting all degree 1 vertices, 
and $\widehat{G}_{a,b}=(\widehat{V}_{a,b}, \widehat{E}_{a,b})$ denote the graph obtained from $G_{a,b}$ 
by contracting edges adjacent to degree 2 vertices, 
one by one repeatedly, until no degree 2 vertex remains.
The degree of $a$, denoted by $\deg(a)$, is the number of edges adjacent to $a$.
We say that $a$ is adjacent to $b$, denoted by $a \adj b$, if there is an edge connecting them.
If they are not adjacent, we denote $a \nadj b$.
We need some notations to count the number of edges of $\widehat{G}_{a,b}$.

\begin{itemize}
\item $V(a)=\{c \in V\ |\ a \adj c \}$  
\item $V_n = \{ c \in V\ |\ {\rm deg}(c)=n \}$
\item $V_n(a)=\{c \in V\ |\ a \adj c,\ \deg(c)=n \}$
\item $V_n(a,b)=V_n(a) \cap V_n(b)$
\item $V_Y(a,b)=\{c \in V\ |\ \exists \ d \in V_3(a,b) \ \mbox{such that} \ c \in V_3(d) \setminus \{a,b\}\}$
\item $E(a)=\{ e \in E\ |\ e \,\, {\rm is \,\, adjacent \,\, to}\,\, a \}$
\item $E(V') = \cup_{a \in V'} E(a)$ for a subset $V'$ of $V$
\end{itemize}

In $G \setminus \{a,b\}$, each vertex of $V_3(a,b)$ has degree 1.
Also each vertex of $V_3(a), V_3(b)$, (but not of $V_3(a,b)$) and $V_4(a,b)$ has degree 2.
Therefore, to derive $\widehat{G}_{a,b}$, all edges adjacent to $a,b$ and $V_3(a,b)$ are removed from $G$,
followed by contracting one of the remaining two edges adjacent to each vertex of
$V_3(a)$, $V_3(b)$, $V_4(a,b)$ and $V_Y(a,b)$ as shown in Figure~\ref{fig5}(a).
We have the following equation counting the number of edges of $\widehat{G}_{a,b}$,
called the {\em count equation\/};
$$|\widehat{E}_{a,b}| = 22 - |E(a)\cup E(b)| - (|V_3(a)|+|V_3(b)|-|V_3(a,b)|+|V_4(a,b)|+|V_Y(a,b)|).$$

For short, we write $NE(a,b) = |E(a)\cup E(b)|$ and $NV_3(a,b) = |V_3(a)|+|V_3(b)|-|V_3(a,b)|$, 
which is non-negative.
Note that the derivation of $\widehat{G}_{a,b}$ must be handled in slightly different way
when there is a vertex $c$ in $V$ such that more than one vertex of $V(c)$ is contained
in $V_3(a,b)$ as in Figure~\ref{fig5}(b).
In this case we usually delete or contract more edges even though $c$ is not in $V_Y(a,b)$.

\begin{figure}[h]
\includegraphics{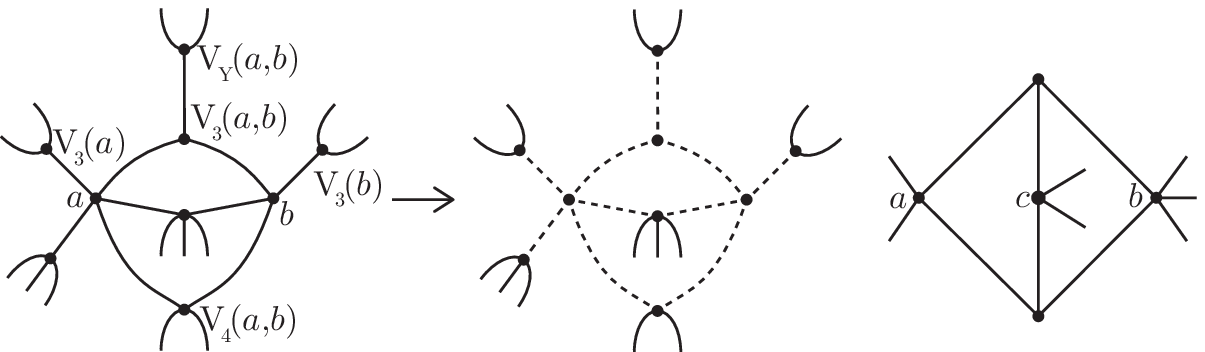}\\
\vspace{4mm}
\hspace{22mm}(a)\hspace{58mm} (b)
\caption{Derivation of $\widehat{G}_{a,b}$}
\label{fig5}
\end{figure}

A graph is called 2-{\em apex\/} if it can be made planar by deleting two vertices.
It is known that if $G$ is 2-apex, then it is not intrinsically knotted~\cite{BBFFHL, OT}.
Thus we have the following proposition.

\begin{proposition} \label{prop:planar}
If $\widehat{G}_{a,b}$ has one of the following conditions, then $G$ is not intrinsically knotted.
\begin{itemize}
\item[(1)] $|\widehat{E}_{a,b}| \leq 8$,
\item[(2)] $\widehat{G}_{a,b}$ consists of nine edges and is not homeomorphic to $K_{3,3}$, or
\item[(3)] $\widehat{G}_{a,b}$ consists of ten edges with a 2-cycle, and does not contain $K_{3,3}$.
\end{itemize}
\end{proposition}

Since the intrinsically knotted graphs with 21 edges are known,
it is sufficient to consider simple and connected graphs having no vertex of degree 1 or 2.
Especially, we will consider graphs $G$ with 22 edges which are triangle-free 
as mentioned in the introduction.
In our process, we pick particular two vertices $a$ and $b$ of $G$
so that $\widehat{G}_{a,b}$ has least possible number of edges, usually less than or equal to 10.
Finally, Proposition~\ref{prop:planar} shows whether $G$ is intrinsically knotted or not. 

In this paper, we are interested in graphs with 22 edges having a single degree 5 vertex $a$
and the remaining degree three or four vertices. 
We distinguish into four cases for $(|V_4|, |V_3|)$; $(0,13)$, $(3,9)$, $(6,5)$ or $(9,1)$. 
We need the followings for consistent use of notations.
Henceforth, $b_i$ and $c_j$ denote the vertices of $V_4(a)$ and $V_3(a)$ respectively.
Let $\overline{V}(a)$ denote the set of vertices which have distance at least 2 from $a$
and $\overline{E}(a)$ denote the set of edges, called {\em extra edges\/}, of the induced subgraph of $G$ 
whose vertex set is $\overline{V}(a)$.
\begin{itemize}
\item $\overline{V}(a) = \{d_1, \dots, d_n\}$ \ for $n=|\overline{V}(a)|$
     with $\deg(d_i) \geq \deg(d_{i+1})$
\item $\overline{E}(a) = \{\overline{e}_1, \dots, \overline{e}_m\}$ \ for $m=|\overline{E}(a)|$
\end{itemize}
Note that we use $d'_i$ instead of $d_i$ if there is no information about their degree.

\section{$(0,13)$, $(3,9)$ and $(9,1)$ types} \label{sec:pro1}

In this section we consider the three types and 
find Cousin 29 of the $K_{3,3,1,1}$ family from the $(3,9)$ type,
which is the only intrinsically knotted graph in these types.

\subsection{$(0,13)$ type} \hspace{1cm}

In this case, $G$ consists of a single degree 5 vertex $a$ and thirteen degree 3 vertices.
Among degree 3 vertices which are not contained in $V(a)$,
one can always choose a vertex $d$ such that $|V_3(a,d)| \leq 2$.
This implies $NE(a,d) = 8$ and $NV_3(a,d) \geq 6$.
Therefore $|\widehat{E}_{a,d}| \leq 8$,
and so $G$ is not intrinsically knotted by Proposition~\ref{prop:planar}.

\subsection{$(3,9)$ type} \hspace{1cm}

We divide into four sub-cases according to the number of the vertices in $V_4(a)$.
First, assume that $V_4(a)$ is empty and choose any degree 4 vertex $d$.
Then $|V_3(a)|=5$ and $NE(a,d)=9$.
The count equation implies that $|\widehat{E}_{a,d}| \leq 8$.

Suppose that $|V_4(a)|=1$ and $b$ denotes the unique degree 4 vertex in $V_4(a)$.
If $|V_3(b)| = 3$, then $NE(a,b)=8$ and $NV_3(a,b)=7$, yielding $|\widehat{E}_{a,b}| \leq 7$.
For otherwise, $V(b)$ has a degree 4 vertex $d$.
Then $NE(a,d)=9$, $|V_3(a)|=4$ and $|V_4(a,d)|=1$, yielding $|\widehat{E}_{a,d}| \leq 8$.

Suppose that $|V_4(a)|=2$, say $b_1$ and $b_2$.
If $|V_3(b_1)|=3$ (or similarly $|V_3(b_2)|=3$), then $|\widehat{E}_{a,b_1}| \leq 8$ as above.
Thus both $V(b_1)$ and $V(b_2)$ must contain the remaining degree 4 vertex, say $d$.
Therefore $NE(a,d)=9$, $|V_3(a)|=3$ and $|V_4(a,d)|=2$, yielding $|\widehat{E}_{a,d}| \leq 8$.

For the last case, $|V_4(a)|=3$, say $b_1$, $b_2$ and $b_3$.
There is a vertex $d_1$, except $a$, adjacent to all $b_1$, $b_2$ and $b_3$.
For otherwise, there are two vertices among them, say $b_1$ and $b_2$, so that $|V_3(b_1,b_2)| \leq 1$.
Therefore $NE(b_1,b_2)=8$, and $NV_3(b_1,b_2)+|V_Y(b_1,b_2)|=6$, yielding
$|\widehat{E}_{b_1,b_2}| \leq 8$.
If there is a vertex $d_2$ in $V_3(b_1,b_2)$ other than $d_1$,
then $|\widehat{E}_{a,b_3}| \leq 9$ and 
furthermore $\widehat{G}_{a,b_3}$ contains a 3-cycle $(b_1 b_2 d_2)$
or a 2-cycle $(b_1 b_2)$ when $d_2 \adj b_3$.
Thus $\widehat{G}_{a,b_3}$ is not homeomorphic to $K_{3,3}$,
and so $G$ is not intrinsically knotted by Proposition~\ref{prop:planar}.
Therefore each of six vertices in $\overline{V}(a)$ excluding $d_1$
is adjacent to exactly one of $b_1$, $b_2$ and $b_3$ as in Figure~\ref{fig:29} (a).

Now we consider $\widehat{G}_{a,d_1}$.
Since $|\widehat{E}_{a,d_1}| \leq 9$,
it remains to consider the case that $\widehat{G}_{a,d_1}$ is homeomorphic to $K_{3,3}$.
The six edges of $\widehat{G}_{a,d_1}$ which do not touching vertices $b_1$, $b_2$ and $b_3$
are denoted by $e_1,\dots,e_6$ as in the figure.
Because of the symmetry of $K_{3,3}$, we may assume that $c_1$ is attached to the middle point of $e_1$.
If $c_2$ is not attached to a middle point of one of $\{e_3, e_4, e_5 \}$, 
then $|\widehat{E}_{b_1,b_2}| \leq 9$ and $\widehat{G}_{b_1,b_2}$ has a 3-cycle as in Figure~\ref{fig:29} (b), 
so $\widehat{G}_{b_1,b_2}$ is not homeomorphic to $K_{3,3}$.
Similarly if $c_2$ is not attached to a middle point of one of $\{ e_2,e_3,e_6 \}$ 
then $\widehat{G}_{b_2,b_3}$ has at most nine edges with a 3-cycle.
Therefore $c_2$ must be attached to the middle point of $e_3$.
This graph $G$ is indeed Cousin 29 of the $K_{3,3,1,1}$ family as drawn in Figure~\ref{fig:3eg}.

\begin{figure}[h]
\includegraphics{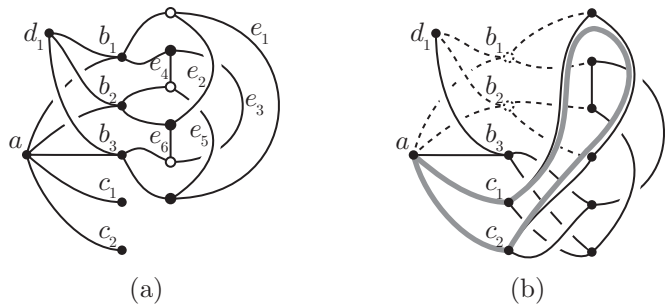}
\caption{Construction of $G$}
\label{fig:29}
\end{figure}

\subsection{$(9,1)$ type} \hspace{1cm}

Note that there are two or three extra edges.
If there is an extra edge $\overline{e}_1$ which is not connected to the other extra edges,
then we may say that $\overline{e}_1$ connects two degree 4 vertices in $\overline{V}(a)$.
Furthermore the rest six edges adjacent to these two vertices must be adjacent to the vertices in $V(a)$,
and so there is a 3-cycle containing $\overline{e}_1$.
For otherwise, all extra edges are connected
and so they miss at least one degree 4 vertex $d_1$ in $\overline{V}(a)$.
Now $|\widehat{E}_{a,d_1}| \leq 9$ since $V(a)$ contains all vertices of $V(d_1)$.
Since $\widehat{G}_{a,d_1}$ still has a degree 4 vertex in $\overline{V}(a)$,
it is not homeomorphic to $K_{3,3}$.

\section{$(6,5)$ type} \label{sec:pro2}

To handle graphs with this type, we distinguish into several cases according to the number of the set $V_4(a)$.

\subsection{$|V_4(a)| \leq 2$ cases} \hspace{1cm}

Two cases $|V_4(a)| = 0$ and $1$ can be handled as in the proof of
the two cases $|V_4(a)| = 0$ and $1$ of the (3,9)-type in Section~\ref{sec:pro1}.

Suppose that $|V_4(a)|=2$, say $b_1$ and $b_2$.
As mentioned in Section~\ref{sec:ter}, $c_1$, $c_2$ and $c_3$ denote the vertices of $V_3(a)$,
and $\overline{V}(a)$ consists of degree 4 vertices $d_1,\dots,d_4$ and 
degree 3 vertices $d_5$ and $d_6$. 
If $V(b_1) \cap V(b_2)$ contains $d_1$, $|\widehat{E}_{a,d_1}| \leq 8$,
and it contains $d_5$, $\widehat{G}_{a,d_5}$ has at most nine edges with a degree 4 vertex.
Therefore we may say that $V(b_1) \cap V(b_2)$ is empty.
If $d_5,d_6 \in V(b_1)$, $\widehat{G}_{a,b_1}$ has at most nine edges with a degree 4 vertex.
Without loss of generality, we further assume that $d_1,d_2,d_5 \in V(b_1)$ and $d_3,d_4,d_6 \in V(b_2)$.
If $d_1$ and $d_5$ are connected by an extra edge, then $|\widehat{E}_{a,d_1}| \leq 8$.
This argument implies that 
the five extra edges connect five pairs of vertices $\{d_1,d_3\},\{d_1,d_4\},\{d_2,d_3\},\{d_2,d_4\},\{d_5,d_6\}$
to avoid creating an unexpected 3-cycle.
Suppose that some $c_i$ is adjacent to two vertices of degree 3 and 4 in $\overline{V}(a)$, say $d_1$ and $d_5$.
Then $|\widehat{E}_{a,d_1}| \leq 8$ since $V_Y(a,d_1)$ is not empty.
It means that one of $c_i$'s is adjacent to both vertices $d_5$ and $d_6$ producing an unexpected 3-cycle.
Therefore $G$ is not intrinsically knotted in this case.

\subsection{$|V_4(a)| = 3$ case} \hspace{1cm}

In this case, we have exactly two triangle-free intrinsically knotted graphs 
Cousin 97 of the $E_9+e$ family (Figure~\ref{fig:3eg}) and $U_{12}$ (Figure~\ref{fig:new}).

The set $V(a)$ consists of degree 4 vertices $b_1, b_2, b_3$ and degree 3 vertices $c_1, c_2$, and
the set $\overline{V}(a)$ consists of degree 4 vertices $d_1, d_2, d_3$ and degree 3 vertices $d_4, d_5, d_6$. 

\begin{claim} \label{clm:con}
If the three degree 4 vertices of $V_4(a)$ are adjacent to 
a vertex $d$ of $\overline{V}(a)$ in common,
then $G$ is not intrinsically knotted.
\end{claim}

\begin{proof}
Consider $\widehat{G}_{a,d}$.
Then $|\widehat{E}_{a,d}| \leq 9$ and there is at least one degree 4 vertex in $\widehat{G}_{a,d}$.
\end{proof}

Since $|\overline{E}(a)|=4$ and $|\overline{V}(a)|=6$, 
we distinguish the induced subgraph of $G$ whose vertices set is $\overline{V}(a)$ into six sub-cases 
according to the connection of four extra edges of $\overline{E}(a)$ as drawn in Figure~\ref{fig:6type}. 
Mention that graphs of Type I and II have vertices of degree 4 and 3, respectively,
and the other graphs have vertices of degree at most 2.

\begin{figure}[h]
\includegraphics{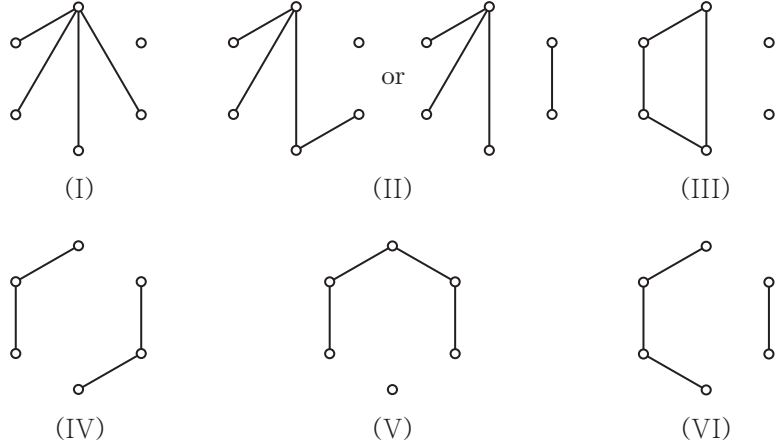}
\caption{Six types according to the connection of $\overline{E}(a)$}
\label{fig:6type}
\end{figure}

\subsubsection{Type I} \hspace{1cm}

Let $d_1$ be the degree 4 vertex where the four extra edges are adjacent.
If $|V_3(d_1)|=3$ then $|\widehat{E}_{a,d_1}| \leq 8$.
So we may assume that $V(d_1)$ consists of two degree 4 vertices $d_2, d_3$ and
two degree 3 vertices $d_4, d_5$.
By Claim~\ref{clm:con}, each of $d_2$ and $d_3$ is adjacent to at least one of $c_1$ or $c_2$.
If $V_3(d_2,d_3) = \{c_1,c_2\}$ then $\widehat{G}_{a,d_1}$ has at most nine edges with a 2-cycle $(d_2 d_3)$.
Without loss of generality, we assume that $V(d_2) = \{b_1, b_2, c_1, d_1\}$.
If $c_1$ is adjacent to a degree 3 vertex in $\overline{V}(a)$,
$|\widehat{E}_{a,d_2}| \leq 8$ because $V_Y(a,d_2)$ is not empty.
Therefore $c_1 \adj d_3$.

Since $|\widehat{E}_{a,d_1}| \leq 9$, we only need to handle the case that $\widehat{G}_{a,d_1}$ is $K_{3,3}$.
In $G_{a,d_1}$ drawn as Figure~\ref{fig:42_12} (a),
there are four dotted degree 2 vertices and six degree 3 vertices
which are colored by black and white related to the vertex partition of $K_{3,3}$.  
Furthermore four dotted edges are already assigned.
The edge of $K_{3,3}$ connecting $d_3$ and $d_6$ inevitably passes through degree 2 vertex $c_2$.
Similarly two edges of $K_{3,3}$ from $b_3$ to $b_1$ and $b_2$ pass through degree 2 vertices $d_4$ and $d_5$.
Draw the remaining edges.
Finally, we obtain the graph $G$ by restoring vertices $a$ and $d_1$ and the adjacent nine edges.
Check the graph $\widehat{G}_{a,d_3}$ which has nine edges with a 3-cycle $(b_1 b_2 d_2)$,
showing that $G$ is not intrinsically knotted.

At this point, we introduce the {\em restoring method\/} 
which is the graph-construction technique we used in the previous paragraph.
This method will be used frequently in the rest of the proof. \\

\noindent {\bf Restoring method.\/} 
The specific condition of the degree information of all vertices and the connection information of some edges is given.
We choose two vertices $a$ and $d$ in $G$ such that $\widehat{G}_{a,d}$ has at most nine edges.
Note that if $\widehat{G}_{a,d}$ is not $K_{3,3}$, it should be 2-apex.
The purpose of the restoring method is to select non 2-apex graphs among all graphs satisfying the given condition. 
By this reason, we only need to consider the case that it is $K_{3,3}$. 
$$ \widehat{G}_{a,d} \ (=K_{3,3}) \ \ \rightarrow \ \ G_{a,d} \ \ \rightarrow \ \ G $$

Now we consider the graph $G_{a,d}$ which consists of six degree 3 vertices and some degree 2 vertices.
Recall that $\widehat{G}_{a,d}$ is obtained from $G_{a,d}$ by replacing every maximal path 
whose interior has only degree 2 vertices by an edge.
Since six degree 3 vertices of $G_{a,d}$ are the vertices of $K_{3,3}$,
we may assign the usual $K_{3,3}$ bi-partition of these vertices from the given specific condition.
Assign the remaining edges of $G_{a,b}$ so as to preserve $\widehat{G}_{a,d}$ to be $K_{3,3}$.
Finally, we obtain the graph $G$ by restoring vertices $a$ and $d$,
and the removed edges.

\begin{figure}[h]
\includegraphics{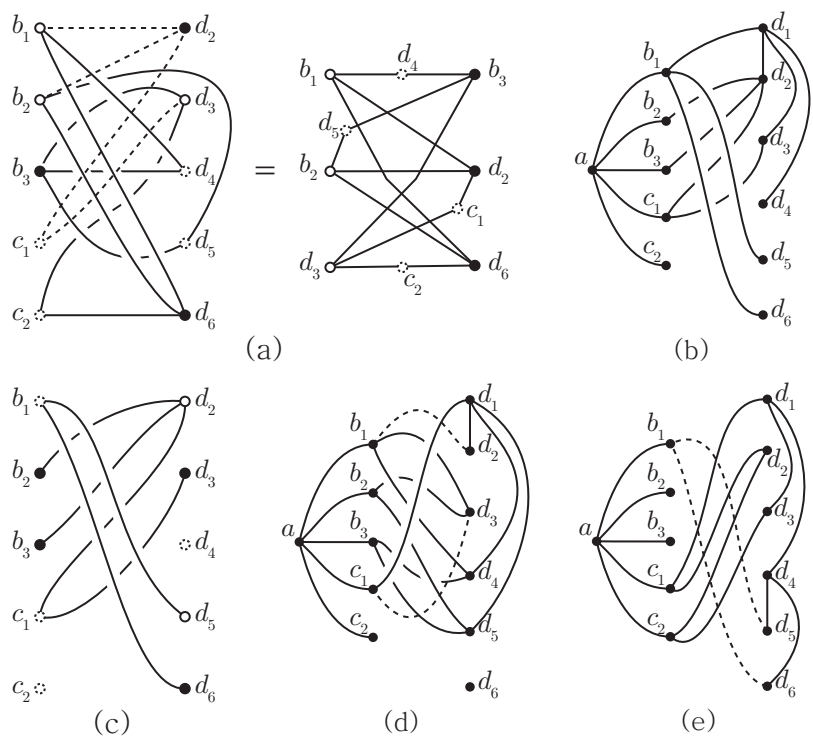}
\caption{Types I and II}
\label{fig:42_12}
\end{figure}

\subsubsection{Type II} \hspace{1cm}

Let $d'$ be the degree vertex where the three extra edges are adjacent.
First consider the case that $d'$ has degree 4 (use $d_1$ instead of $d'$ as usual) and
we may assume the fourth edge adjacent to $d_1$ is also adjacent to either $b_1$ or $c_1$.
In the former case, we may assume that the three extra edges adjacent to $d_1$
are adjacent to the rest two degree 4 vertices $d_2,d_3$ and one degree 3 vertex $d_4$.
For otherwise, $|V_3(d_1)| \geq 2$ implying $|\widehat{E}_{a,d_1}| \leq 8$.
Since $G$ is a triangle-free graph, the vertex $b_1$ is adjacent to the remaining degree 3 vertices $d_5,d_6$.
Without loss of generality, the fourth extra edge is not adjacent to $d_2$.
If $d_2$ is adjacent to two degree 3 vertices $c_1, c_2$, then $\widehat{G}_{b_1,d_2}$ has at most nine edges with a degree 4 vertex $a$.
So, it is not homeomorphic to $K_{3,3}$.
So we assume that $d_2$ is adjacent to two degree 4 vertices $b_2, b_3$ and a degree 3 vertex $c_1$.
Furthermore $c_1 \adj d_3$ as drawn in Figure~\ref{fig:42_12} (b),
for otherwise $|\widehat{E}_{a,d_2}| \leq 8$ because $V_Y(a,d_2)$ is not empty.
Then the graph drawn as Figure~\ref{fig:42_12} (c) becomes a subgraph of $\widehat{G}_{a,d_1}$.
But, we cannot assign $K_{3,3}$ bi-partition of degree 3 vertices.

In the latter case so that $d_1 \adj c_1$,
we divide into three cases according to $|V_4(d_1)|$.
If $V_4(d_1)$ is empty, $|\widehat{E}_{a,d_1}| \leq 8$.
If $V_4(d_1)$ has exactly one vertex $d_2$, 
then we may say that $d_4$ and $d_5$ are the remaining degree 3 vertices of $V(d_1)$,
and $d_3$ and $d_6$ are respectively the degree 4 and 3 vertices not in $V(d_1)$.
Now $b_1$, similarly $b_2$ and $b_3$, is adjacent to at least two vertices of $d_3$, $d_4$ and $d_5$,
for otherwise, $b_1$ must adjacent to both $d_2$ and $d_6$ 
and so $\widehat{G}_{b_1,d_1}$ has at most nine edges with a degree 4 vertex $a$.
Without loss of generality, by Claim~\ref{clm:con}, 
$b_1$, $b_2$ and $b_3$ are adjacent to $d_3$, $d_4$ and $d_5$ as in Figure~\ref{fig:42_12} (d). 
Furthermore, by considering $\widehat{G}_{a,d_1}$, $c_1 \nadj d_6$, so $c_1 \adj d_3$;
by considering $\widehat{G}_{a,b_1}$, $b_1 \nadj d_6$, so $b_1 \adj d_2$.
Since $d_6$ must be adjacent to $b_2$ or $b_3$, $\widehat{G}_{b_2,b_3}$ has at most ten edges with a 2-cycle and a 3-cycle.
Therefore, after deleting an edge of the 2-cycle, it is not homeomorphic to $K_{3,3}$.
Lastly, if $|V_4(d_1)|=2$, $V_3(d_1)$ has two vertices which are $c_1$ and $d_4$, say.
Thus $d_4$ is not adjacent to at least one vertex among $b_1$, $b_2$ and $b_3$, say $b_1$.
Then $\widehat{G}_{b_1,d_1}$ has at most nine edges with a degree 4 vertex $a$.

Now consider the case that $d'$ has degree 3 (use $d_4$ instead).
Assume that the three vertices adjacent to $d_4$ are all degree 4 vertices $d_1, d_2, d_3$.
Then there is a vertex $d_i$ among them such that $d_i$ is adjacent to at most one of $c_1$ and $c_2$,
and so $|\widehat{E}_{a,d_i}| \leq 8$ because $NV_3(a,d_1) + |V_4(a,d_1)| \geq 5$.
Assume that exactly two vertices adjacent to $d_4$ are degree 4 vertices $d_1, d_2$
with the assumption that the fourth extra edge is not adjacent to $d_1$.
To be $|\widehat{E}_{a,d_1}| \geq 9$, $d_1$ must be adjacent to both $c_1$ and $c_2$, 
and furthermore $c_1$ is adjacent to either $d_2$ or $d_3$, and $c_2$ is adjacent to the other.
Since $|V_4(a,d_3)| \geq 2$, $\widehat{G}_{a,d_3}$ has at most nine edges with a 3-cycle $(d_1 d_2 d_4)$.
It remains to consider the case that the vertices adjacent to $d_4$ are 
one degree 4 vertex $d_1$ and two degree 3 vertices $d_5, d_6$. 
If $d_2$ and $d_3$ are connected by the fourth extra edge,
the rest six edges adjacent to $d_2$ or $d_3$ are also adjacent to the five vertices adjacent to $a$
producing a 3-cycle.
Thus we may assume that no extra edge is adjacent to $d_2$.
To be $|\widehat{E}_{a,d_2}| \geq 9$, $d_2$ must be adjacent to both $c_1$ and $c_2$, 
and furthermore $c_1$ is adjacent to either $d_1$ or $d_3$, and $c_2$ is adjacent to the other.
Since $|V_4(a,d_3)| \geq 2$, $\widehat{G}_{a,d_3}$ has at most nine edges with a 3-cycle $(d_1 d_2 d_4)$.
Similarly to be $|\widehat{E}_{a,d_3}| \geq 9$, $d_3 \adj d_1$ as in Figure~\ref{fig:42_12} (e).
Now the rest four edges adjacent to $d_5$ or $d_6$ are also adjacent to the three vertices $b_1, b_2, b_3$,
so we may say that $b_1$ is adjacent to both $d_5$ and $d_6$.
Then $\widehat{G}_{a,d_1}$ has at most nine edges with a 2-cycle or a 3-cycle containing $d_5$ and $d_6$.

\subsubsection{Type III} \hspace{1cm}

Let $C$ be the cycle consisting of the four extra edges.
If there is a vertex, say $d_1$, of degree 4 in $\overline{V}(a)$ away from $C$,
vertex $d_1$ must be adjacent to both $c_1$ and $c_2$ by Claim~\ref{clm:con}, 
and furthermore $c_1$ is adjacent to either $d_2$ or $d_3$, and $c_2$ is adjacent to the other
for being $|\widehat{E}_{a,d_2}| \geq 9$.
Let $d_i$ be the remaining vertex of $\overline{V}(a)$ not in $C \cup \{d_1 \}$.
Vertex $d_i$ must be adjacent to all $b_1$, $b_2$ and $b_3$ 
because all edges adjacent to $c_1$ and $c_2$ are occupied.
This contradicts to Claim~\ref{clm:con}. 

Without loss of generality, $C= (d_1 d_2 d_3 d_4)$ where only $d_4$ has degree 3.
By Claim~\ref{clm:con}, each vertex $d_5$ and $d_6$ of degree 3 is adjacent to 
at least one of $c_1$ and $c_2$.
Furthermore $d_1$ (and similarly $d_3$) is adjacent to either $c_1$ or $c_2$,
for otherwise $|V_4(a,d_1)| = 2$ and so $|\widehat{E}_{a,d_1}| \leq 8$.

If $c_1$ adjacent to $d_1$ is adjacent to one of $d_5$ and $d_6$,
then $|\widehat{E}_{a,d_1}| \leq 8$ because of $|V_Y(a,d_1)| = 1$.
Therefore $c_1$ is adjacent to both $d_1$ and $d_3$ as in Figure~\ref{fig:42_34} (a),
and so $\widehat{G}_{a,d_2}$ has at most nine edges with a 3-cycle $(d_1 d_3 d_4)$.

\begin{figure}[h]
\includegraphics{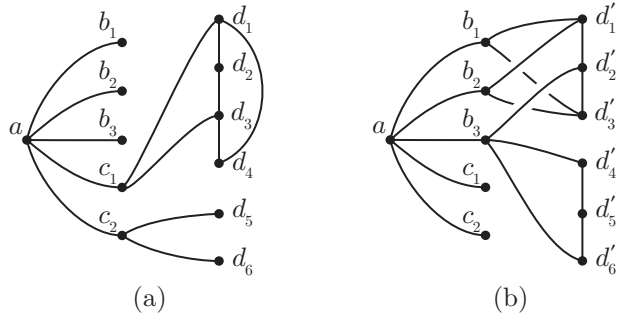}
\caption{Types III and IV}
\label{fig:42_34}
\end{figure}

\subsubsection{Type IV} \hspace{1cm}

By Claim~\ref{clm:con}, 
we may assume $b_1$, $b_2$ and $b_3$ are adjacent to $\overline{V}(a)$ as in Figure~\ref{fig:42_34} (b).
Suppose that $d'_2$ has degree 4 and is adjacent to $c_1$.
By considering $\widehat{G}_{a,d'_2}$, both $d_1$ and $d_3$ have degree 4,
for otherwise $\widehat{G}_{a,d'_2}$ has at most nine edges with a 3-cycle $(d'_4 d'_5 d'_6)$.
Now $c_1$ is adjacent to one vertex $d'_i$ among the degree 3 vertices of $\overline{V}(a)$.
Since $|V_Y(a,d'_2)| = 1$, $\widehat{G}_{a,d'_2}$ has at most nine edges with a 2-cycle.

Now consider that  $d'_2$ has degree 3.
If $d'_1$ has degree 4, then $|\widehat{E}_{a,d'_1}| \leq 8$.
Thus $d'_1$ and similarly $d'_3$ have degree 3.
Therefore $\widehat{G}_{a,d'_2}$ has at most nine edges with a 3-cycle $(d'_4 d'_5 d'_6)$.

\subsubsection{Type V} \hspace{1cm}

Let $d'_1,\dots,d'_5$ denote the five vertices appeared successively on the path 
of length 4 consisting of the edges of $\overline{E}(a)$.
Let $d'_6$ be the remaining vertex.
By Claim~\ref{clm:con}, we may assume that $b_1 \nadj d'_6$.
Then $b_1$ must be adjacent to $d'_1, d'_3, d'_5$ as in Figure~\ref{fig:42_5} (a).
If $d'_1$ has degree 4,
$d'_1$ must be adjacent to both $c_1$ and $c_2$.
For otherwise, $\widehat{G}_{a,d'_1}$ has at most nine edges with a 3-cycle $(d'_3 d'_4 d'_5)$.
Furthermore $d'_2$ has degree 4 because, if not, $|\widehat{E}_{a,d'_1}| \leq 8$.
In this case, at least one of $c_1$ and $c_2$ is adjacent to a degree 3 vertex $d'_i$
or they are adjacent to the remaining degree 4 vertex $d'_j$, $j \neq 1,2$.
In both cases, $\widehat{G}_{a,d'_1}$ has at most nine edges with a 3-cycle $(d'_3 d'_4 d'_5)$
(it may be a loop or 2-cycle).

So we may assume that $d'_1$ (similarly $d'_5$) has degree 3,
and without loss of generality $d'_2$ has degree 4.
We further assume that $d'_3$ has degree 4,
for otherwise, $\widehat{G}_{a,d'_2}$ has at most nine edges with a 3-cycle $(b_1 d'_4 d'_5)$.
If $d'_6$ has degree 4, it must be adjacent to $b_2, b_3, c_1, c_2$.
Moreover $c_1$ is adjacent to either $d'_2$ or $d'_3$, and $c_2$ is adjacent to the other,
for otherwise, $|\widehat{E}_{a,d'_6}| \leq 8$.
Now we assume that $d'_2 \adj b_2$.
Since $d'_4$ has degree 3, $\widehat{G}_{a,d'_3}$ has at most nine edges with a 3-cycle $(b_2 d'_2 d'_6)$ as in Figure~\ref{fig:42_5} (b).
Therefore $d'_6$ has degree 3 and so $d'_4$ has degree 4.

Suppose that $b_2 \nadj d'_6$.
As $b_1$, it must be adjacent to $d'_1, d'_3, d'_5$.
Then $\widehat{G}_{a,d'_3}$ has at most nine edges with a 2-cycle $(d'_1 d'_5)$.
Without loss of generality, $d'_6$ is adjacent to $b_2, b_3, c_1$.
If $c_2 \nadj d'_2$, $|\widehat{E}_{a,d'_2}| \leq 8$ since $|V_4(a,d'_2)| + |V_Y(a,d'_2)| = 2$.
Therefore $c_2 \adj d'_2$, and similarly $c_2 \adj d'_4$.
If $d'_3 \adj b_2$ (or similarly $d'_3 \adj b_3$), $b_2$ must be adjacent to either $d'_1$ or $d'_5$,
say $b_2 \adj d'_1$.
Then $\widehat{G}_{a,d'_2}$ has at most nine edges with a 3-cycle $(b_1 b_2 d'_3)$.
Therefore $d'_3 \adj c_1$ as in Figure~\ref{fig:42_5} (c).

Finally, if $d'_1$ and $d'_5$ are adjacent to the same vertex $b_2$ (or $b_3$),
then $\widehat{G}_{a,d'_3}$ has at most nine edges with a 3-cycle $(b_2 d'_1 d'_5)$.
Without loss of generality, $b_2$ is adjacent to both $d'_1$ and $d'_4$,
and $b_3$ is adjacent to both $d'_2$ and $d'_5$.
Then $G$ is homeomorphic to Cousin 97 of the $E_9+e$ family.

\begin{figure}[h!]
\includegraphics{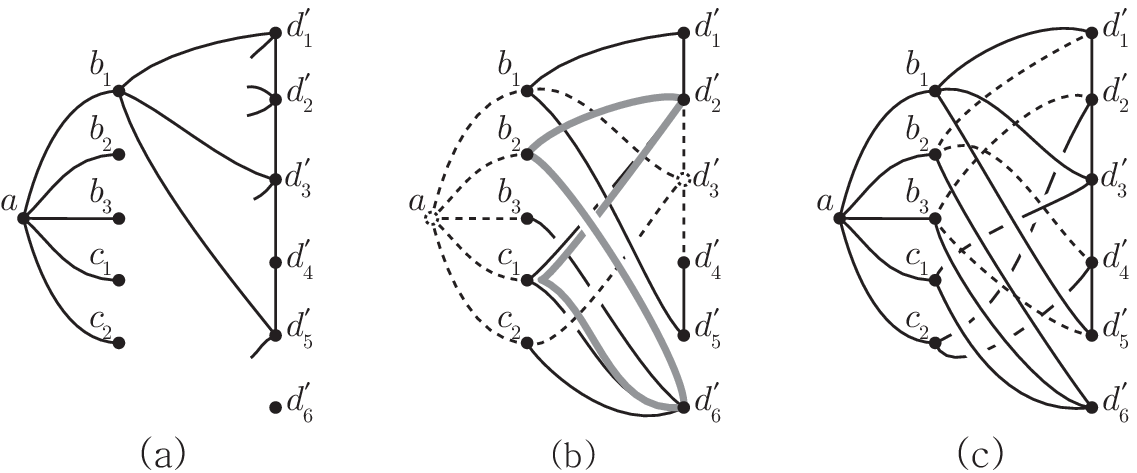}
\caption{Type V}
\label{fig:42_5}
\end{figure}

\subsubsection{Type VI} \hspace{1cm}

Let $d'_1,\dots,d'_4$ denote the four vertices appeared successively on the path 
of length 3 consisting of the edges of $\overline{E}(a)$,
and the remaining vertices $d'_5,d'_6$ are connected by an edge of $\overline{E}(a)$.
Each degree 4 vertex $b_i$ is adjacent to two vertices among $d'_1,\dots,d'_4$ and one of $d'_5,d'_6$.
Consider the number $|V_4(a,d'_2)|+|V_4(a,d'_3)|$ which must be 3 or 2 by Claim~\ref{clm:con}.
Without loss of generality, either case is drawn as Figure~\ref{fig:42_6} (a) or (b), respectively.

\begin{figure}[h!]
\includegraphics{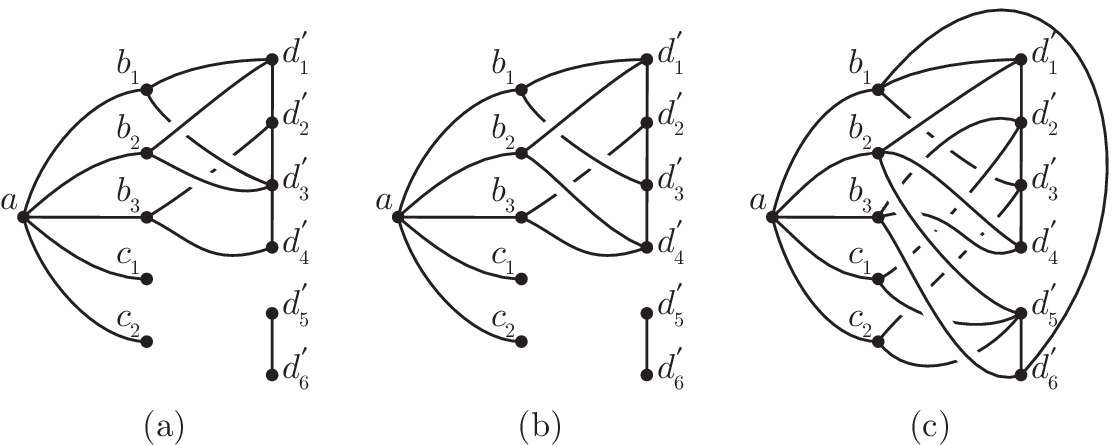}
\caption{Type VI}
\label{fig:42_6}
\end{figure}

For the first case (a), 
$d'_2$ and $d'_4$ have degree 4, for otherwise $|\widehat{E}_{a,d'_3}| \leq 8$.
Furthemore $d'_4$ is adjacent to both $c_1$ and $c_2$ and we may assume that $d'_2 \adj c_1$.
Therefore $\widehat{G}_{a,d'_3}$ has at most nine edges with a 3-cycle $(b_3 d'_2 d'_4)$.

For the second case (b),
If $d'_2$ has degree 3, then $\widehat{G}_{a,d'_1}$ has at most nine edges with a 3-cycle $(b_3 d'_3 d'_4)$.
So $d'_2$ and similarly $d'_3$ have degree 4.
Assume that $d'_2 \adj c_1$ and $d'_3 \adj c_2$.
If $d'_1$ has degree 4, then $d'_1 \adj c_2$.
Since $V_Y(a,d'_2)$ is nonempty, $\widehat{G}_{a,d'_2}$ has at most nine edges with a 3-cycle $(b_1 d'_1 d'_3)$.
So $d'_1$ and similarly $d'_4$ have degree 3.
Assume that $d'_5$ and $d'_6$ have degree 4 and 3, respectively.
Now $d'_5$ is adjacent to both $c_1$ and $c_2$, for otherwise $|\widehat{E}_{a,d'_5}| \leq 8$.
If $d'_5 \adj b_1$, then $\widehat{G}_{a,d'_5}$ has at most nine edges with a 3-cycle $(d'_1 d'_2 d'_3)$.
Finally $d'_6 \adj b_1$ and similarly $d'_6 \adj b_3$, and so $d'_5 \adj b_2$ as drawn in Figure~\ref{fig:42_6} (c).
This graph is a previously unknown intrinsically knotted graph $U_{12}$ in Figure~\ref{fig:new}.

\subsection{$|V_4(a)| = 4$ case} \label{subsec:4} \hspace{1cm}

In this case, we have exactly two triangle-free intrinsically knotted graphs 
Cousin 99 of the $E_9+e$ family (Figure~\ref{fig:3eg}) and $U_{12}'$ (Figure~\ref{fig:new}).

The set $V(a)$ consists of degree 4 vertices $b_1, b_2, b_3, b_4$ and degree 3 vertex $c_1$, and
the set $\overline{V}(a)$ consists of degree 4 vertices $d_1, d_2$ and degree 3 vertices $d_3, d_4, d_5, d_6$. 
We first introduce three useful claims.

\begin{claim} \label{clm:cl2}
$d_1 \nadj d_2$.
\end{claim}

\begin{proof}
Suppose instead $d_1 \adj d_2$.
Without loss of generality, we assume $c_1 \nadj d_1$ since $G$ is a triangle-free graph.
We divide into three cases according to $|V_4(a,d_1)|$ as in Figure~\ref{fig:cl2}.

\begin{figure}[h]
\includegraphics{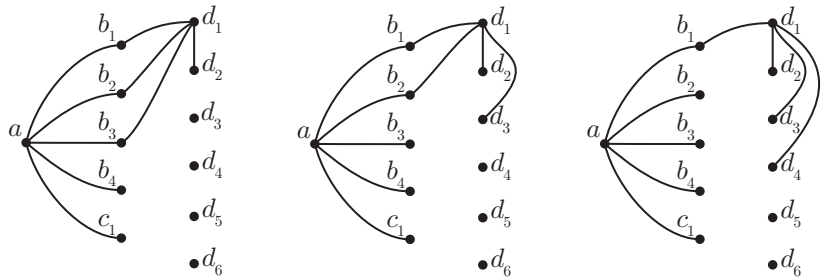}
\caption{Three cases for Claim~\ref{clm:cl2}}
\label{fig:cl2}
\end{figure}

First, assume that $d_1$ is adjacent to $b_1,b_2,b_3$.
Consider $V_3(b_1,b_2)$.
If $V_3(b_1,b_2)$ has two vertices, say $d_3, d_4$,
then $\widehat{G}_{a,d_1}$ has at most nine edges with a 2-cycle $(d_3 d_4)$.
If $V_3(b_1,b_2)$ is empty,
then $\widehat{G}_{b_1,b_2}$ has at most nine edges with a degree 4 vertex $d_2$.
Therefore $|V_3(b_1,b_2)|=1$ and similarly $|V_3(b_1,b_3)|=|V_3(b_2,b_3)|=1$.
Therefore we have two configurations drawn as Figure~\ref{fig:cl21}.

\begin{figure}[h]
\includegraphics{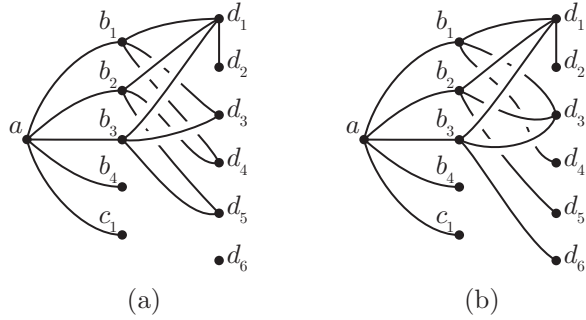}
\caption{Two configurations in the first case}
\label{fig:cl21}
\end{figure}

For Figure~\ref{fig:cl21} (a), $\widehat{G}_{a,d_1}$ has at most nine edges with a 3-cycle $(d_3 d_4 d_5)$.
For Figure~\ref{fig:cl21} (b), 
$|\widehat{E}_{a,d_1}| \leq 9$ and so we only need to handle that $\widehat{G}_{a,d_1}$ is $K_{3,3}$.
In the process of the restoring method,
we construct $G_{a,d_1}$ which is drawn as Figure~\ref{fig:claim2} (a).
In details, there are four dotted degree 2 vertices and six degree 3 vertices
which are colored by black and white related to the vertex partition of $K_{3,3}$.
Note that three vertices $d_4, d_5, d_6$ are in the same partition because of the three dotted edges.
Then there is a unique way to connect $b_4$ and the three black vertices,
and also $d_2$ is connected to these black vertices by the remaining two extra edges
and the path via $c_1$.
Now we obtain the graph $G$ by restoring vertices $a$ and $d_1$ and the adjacent nine edges.
Then $\widehat{G}_{b_3,d_2}$ has eight edges.

\begin{figure}[h]
\includegraphics{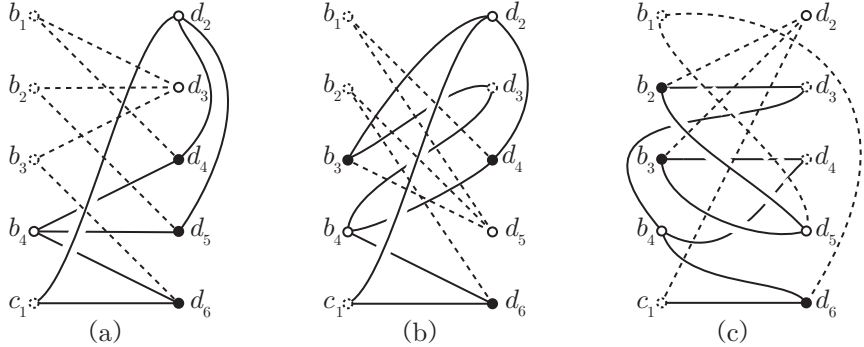}
\caption{Restoring methods for Claim~\ref{clm:cl2}}
\label{fig:claim2}
\end{figure}

Second, assume that $d_1$ is adjacent to $b_1,b_2,d_3$.
If $V_3(b_1,b_2)$ has two vertices, say $d_3, d_4$,
then $\widehat{G}_{a,d_1}$ has at most nine edges with a 2-cycle $(d_3 d_4)$.
We assume without loss of generality that $V_3(b_1)=\{d_4,d_5\}$ and $V_3(b_2)=\{d_5,d_6\}$.
The remaining edge adjacent to $d_5$ must be adjacent to a degree 4 vertex.
For otherwise, $\widehat{G}_{b_1,b_2}$ has at most nine edges with a degree 4 vertex $d_2$.
If $d_5 \adj d_2$, then $|V_4(a,d_2)|+|V_Y(a,d_2)|=2$, 
and so $\widehat{G}_{a,d_2}$ has at most nine edges with a 3-cycle $(b_1 b_2 d_1)$.
Therefore $d_5 \adj b_3$ (or similarly for $d_5 \adj b_4$).
Following the restoring method,
we construct $G_{a,d_1}$ as drawn in Figure~\ref{fig:claim2} (b).
In this case, $b_3, d_4, d_6$ are in the same partition
and the rest of the construction is similar to the previous case.
From $G$ obtained by restoring the vertices and edges,
we have $\widehat{G}_{b_2,d_2}$ with nine edges and a degree 4 vertex~$a$.

Third, assume that $d_1$ is adjacent to $b_1,d_3, d_4$,
and so $V_3(b_1)=\{d_5,d_6\}$.
If $d_2$ is adjacent to $b_2, b_3, b_4$, we already handle in the first case of this proof.
Therefore we may assume that $d_2$ is adjacent to $b_2, b_3, c_1$.
Following the restoring method,
we construct $G_{a,d_1}$ as drawn in Figure~\ref{fig:claim2} (c).
In this case, $b_2, b_3$ are in the same partition and $d_5, d_6$ are in different partitions.
We assume that $b_2, b_3, d_6$ are in the same partition,
and similarly construct the rest.
From the restored graph $G$,
we get $\widehat{G}_{b_2,b_4}$ with nine edges and a degree 4 vertex $d_2$.
\end{proof}

\begin{claim} \label{clm:cl3}
At least one extra edge is adjacent to $d_i$, $i=1,2$.
\end{claim}

\begin{proof}
Suppose instead that $|V(a) \cap V(d_1)|=4$.
If $|V_4(a,d_1)|=4$ then $|\widehat{E}_{a,d_1}| \leq 8$.
So we assume that $d_1$ is adjacent to $b_1,b_2,b_3,c_1$.
If $V_Y(a,d_1) $ is not empty, then $|\widehat{E}_{a,d_1}| \leq 8$.
Thus $c_1 \adj d_2$ by an edge $e$.
Let $d'_1, \dots, d'_5$, denote five vertices of $\overline{V}(a)$ excluding $d_1$,
and the three extra edges are arranged by three ways (I), (II) and (III) as described in Figure~\ref{fig:cl3}.
Note that we do not distinguish $d_2$ from the other four degree 3 vertices.

\begin{figure}[h]
\includegraphics{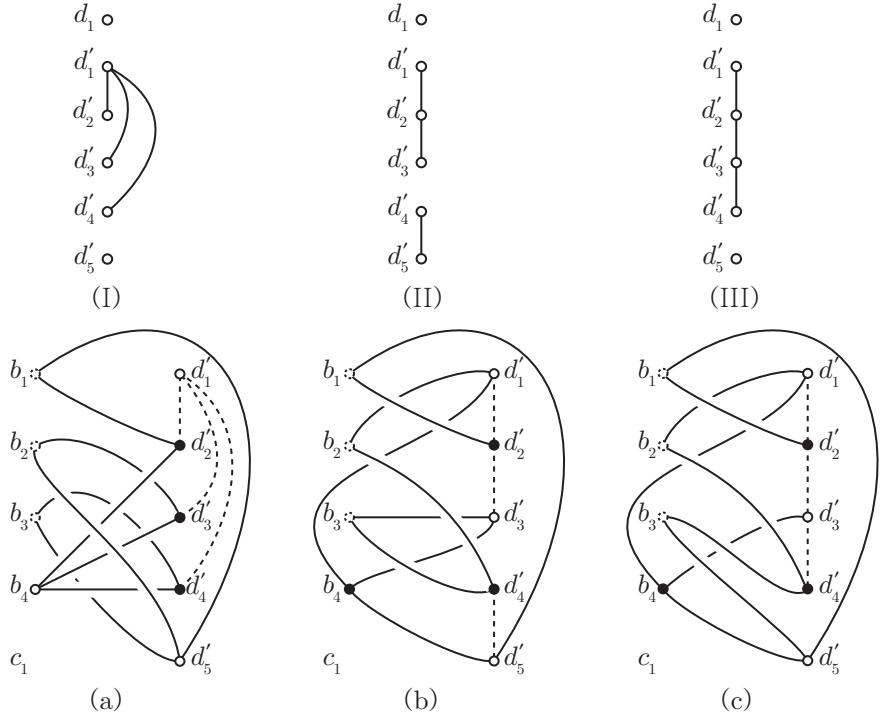}
\caption{Three cases according to the way $\overline{E}(H)$ is connected in Claim~\ref{clm:cl3}}
\label{fig:cl3}
\end{figure}

Following the restoring method,
in three cases (I), (II) and (III), 
we easily obtain $G_{a,d_1}$ as drawn in Figure~\ref{fig:cl3} (a), (b) and (c), respectively.
Now we obtain the graph $G$ by restoring the nine edges adjacent to $a$ and $d_1$
and one more edge $e$ connecting $c_1$ and one among $d'_1, \dots, d'_5$.
In the figure (a), 
if $c_1 \adj d'_1$, then $\widehat{G}_{b_1,d'_1}$ has at most nine edges with degree 4 vertex $a$.
If $c_1 \adj d'_2$, then $\widehat{G}_{a,d'_2}$ has at most nine edges with a 2-cycle $(d'_3 d'_4)$.
If $c_1 \adj d'_5$, then $\widehat{G}_{b_1,b_2}$ has at most ten edges 
with two distinct 2-cycles $(b_3 c_1)$ and $(b_4 d'_1)$.
In the figure (b),
if $c_1 \adj d'_1$ (and similarly $c_1 \adj d'_3$), then $\widehat{G}_{a,d'_1}$ has at most nine edges 
with a 3-cycle $(b_3 d_1 d'_4)$.
If $c_1 \adj d'_2$, then $\widehat{G}_{b_2,b_3}$ has at most nine edges with a 3-cycle $(a b_1 c_1)$.
If $c_1 \adj d'_4$, then $\widehat{G}_{a,d'_1}$ has at most ten edges 
with a 2-cycles $(d_1 d'_4)$ and a 3-cycle $(b_3 d_1 d'_4)$.
If $c_1 \adj d'_5$, then $\widehat{G}_{a,d'_5}$ has at most nine edges with a 3-cycle $(b_2 b_3 d_1)$.
In the figure (c),
if $c_1 \adj d'_i$, $i=1, \dots, 5$, 
then the reader easily check that $\widehat{G}_{a,d'_i}$ has at most nine edges with a 3-cycle in each case.
\end{proof}

\begin{claim} \label{clm:cc}
$c_1$ is adjacent to both $d_1$ and $d_2$.
\end{claim}

\begin{proof}
Suppose instead that $c_1 \nadj d_1$.
Then $|\widehat{E}_{a,d_1}| \leq 8$ because $d_1 \nadj d_2$.
\end{proof}

Since $|\overline{E}(a)|=3$ and $|\overline{V}(a)|=6$, by Claims~\ref{clm:cl2} and \ref{clm:cl3}, 
we distinguish the induced subgraph $H$ of $G$ whose vertices set is $\overline{V}(a)$ into seven sub-cases 
according to the connection of three extra edges of $\overline{E}(a)$ as drawn in Figure~\ref{fig:7type}. 
In details, if $H$ has a vertex of degree 3, $d_1$ and $d_2$ must be vertices of degree 1, 
so we may say that $d_3$ is adjacent to $d_1,d_2,d_4$ as in Type I.
If $H$ has all disconnected three edges, it has Type II.
If $H$ has consecutive three edges, it has two Types III and IV because $d_1 \nadj d_2$.
Finally, $H$ has three Types V, VI and VII when only two edges are connected.
Recall that $c_1$ is adjacent to $d_1, d_2$.
Without loss of generality we assume that $d_1$ is adjacent to $b_1, b_2$.

\begin{figure}[h]
\includegraphics{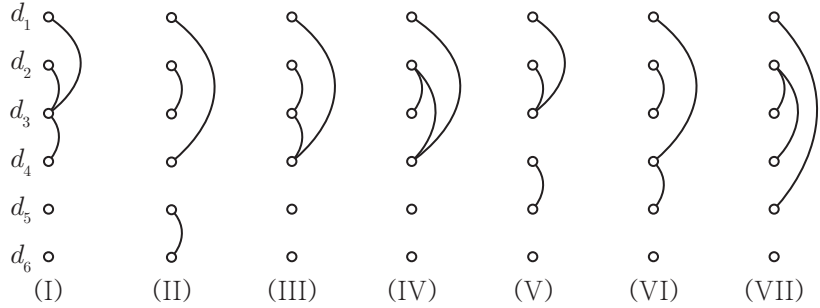}
\caption{Seven types according to the connection of $\overline{E}(a)$}
\label{fig:7type}
\end{figure}

For Type I, we claim that $d_2$ is adjacent to $b_3, b_4$.
Suppose instead that $d_2 \adj b_1$.
We have two choices for the remaining edge adjacent to $b_1$.
If $b_1 \adj d_4$, $\widehat{G}_{a,d_4}$ has at most ten edges, 
three edges among them connect $d_1$ and $d_2$.
If $b_1 \adj d_5$, $\widehat{G}_{a,d_5}$ has at most ten edges 
with a 2-cycle $(d_1 d_2)$ and a 3-cycle $(d_1 d_2 d_3)$.
Following the restoring method,
we construct $G_{a,d_1}$ as drawn in Figure~\ref{fig:TypeI-II} (a).
In this case, $b_3, b_4, d_4$ are in the same partition.
In the restored $G$, we have $\widehat{G}_{b_1,b_2}$ with nine edges and a 2-cycle $(b_3 b_4)$.

For Type II, $d_4$ is adjacent to $b_3, b_4$ since $G$ is triangle-free.
Following the restoring method,
we construct $G_{a,d_1}$ as drawn in Figure~\ref{fig:TypeI-II} (b).
In this case, without loss of generality, we assume that $d_2, b_3, d_5$ are in the same partition.
Then the restored $G$ is homeomorphic to Cousin 99 of the $E_9+e$ family
as drawn in Figure~\ref{fig:TypeI-II} (c).

\begin{figure}[h]
\includegraphics{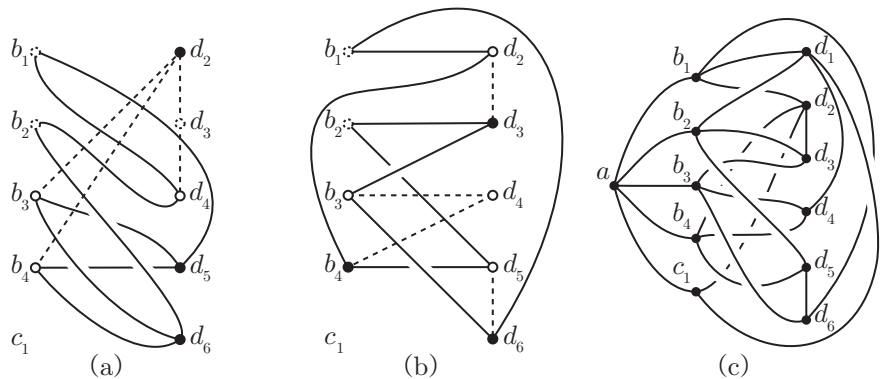}
\caption{Types I and II}
\label{fig:TypeI-II}
\end{figure}

For Type III, we consider the remaining edge adjacent to $d_3$.
If $d_3 \adj b_1$ (or similarly $b_2$), 
$\widehat{G}_{a,d_2}$ has at most nine edges with a 3-cycle $(b_1 d_1 d_4)$.
So $d_3 \adj b_3$ (or $b_4$), 
and so $\widehat{G}_{b_3,d_1}$ has at most nine edges with a degree 4 vertex~$a$.

For Type IV, we say that $d_4 \adj b_3$ since $G$ is triangle-free.
If $b_4 \nadj d_2$, $NV_3(b_4,d_1)=5$ and 
so $\widehat{G}_{b_4,d_1}$ has at most nine edges with a degree 4 vertex $a$.
Therefore $b_4$ is adjacent to $d_2, d_5, d_6$.
If $b_3 \adj d_3$, $\widehat{G}_{a,d_1}$ has at most nine edges with a 3-cycle $(b_3 d_2 d_3)$.
Thus $b_3$ is adjacent to $d_5, d_6$.
But, $\widehat{G}_{a,d_2}$ has at most nine edges with a 3-cycle $(b_3 d_5 d_6)$.

For Type V, we assume that $d_3 \adj b_3$ and further assume that $b_3$ is adjacent to $d_5, d_6$
since $G$ is triangle-free.
Following the restoring method,
we construct $G_{a,d_1}$ as drawn in Figure~\ref{fig:TypeV-VII} (a).
In this case, $d_2, d_5, d_6$ are in the same partition.
Then $\widehat{G}_{a, d_2}$ has nine edges with a 3-cycle $(b_3 d_5 d_6)$.

For Type VI, we assume that $d_4 \adj b_3$ and so $b_3 \adj d_6$.
Furthermore $b_4$ is adjacent to $d_5, d_6$.
If $d_2$ is adjacent to both $b_3$ and $b_4$,
then $\widehat{G}_{a, d_2}$ has nine edges with a 3-cycle $(d_4 d_5 d_6)$.
So $d_2$ is adjacent to $b_1$ or $b_2$.
Following the restoring method,
we construct $G_{a,d_1}$ as drawn in Figure~\ref{fig:TypeV-VII} (b).
In this case, $d_2, b_3, b_4$ are in the same partition.
Then $\widehat{G}_{a, d_2}$ has nine edges with a 3-cycle $(b_3 b_4 d_6)$.

For Type VII, $d_5$ should be adjacent to $b_3, b_4$.
Following the restoring method,
$G_{a,d_1}$ is constructed as drawn in Figure~\ref{fig:TypeV-VII} (c).
In this case, $d_2, b_3, d_6$ are in the same partition.
Then the restored $G$ is as drawn in Figure~\ref{fig:TypeV-VII} (d).
This graph is a previously unknown intrinsically knotted graph $U'_{12}$ in Figure~\ref{fig:new}.

\begin{figure}[h]
\includegraphics{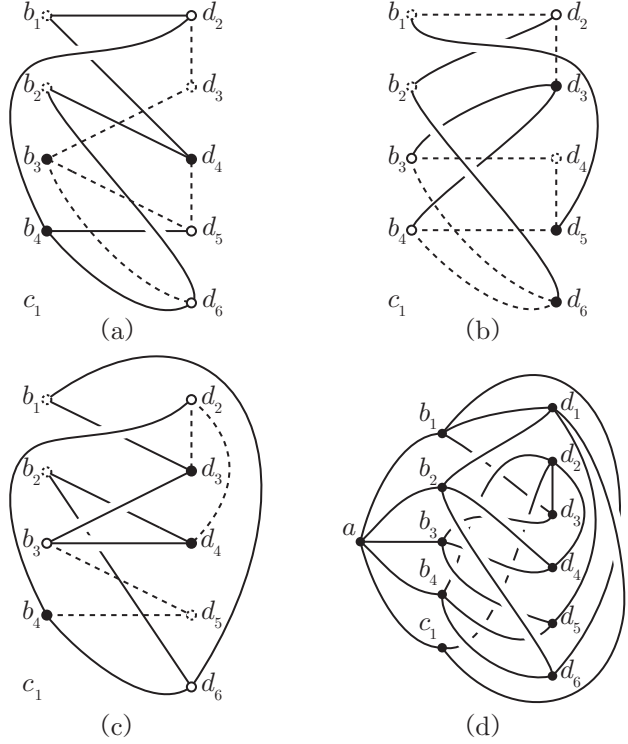}
\caption{Types V, VI and VII}
\label{fig:TypeV-VII}
\end{figure}

\subsection{$|V_4(a)| = 5$ case} \hspace{1cm}

The set $V(a)$ consists of degree 4 vertices $b_1, \dots, b_5$, and
the set $\overline{V}(a)$ consists of a degree 4 vertex $d_1$ and degree 3 vertices $d_2, \dots, d_6$. 
Now we distinguish into three cases according to the number of vertices of $V(a) \cap V(d_1)$.
First suppose that $|V(a) \cap V(d_1)|=2$.
Without loss of generality, we may assume that $d_1$ is adjacent to $b_1, b_2, d_2, d_3$.
We say that $b_1$ is adjacent to $d_4, d_5$.
If $b_2$ is adjacent to both $d_4$ and $d_5$,
$\widehat{G}_{a, d_1}$ has nine edges with a 2-cycle $(d_4 d_5)$,
so we assume that $b_2$ is adjacent to $d_5, d_6$.
Following the restoring method,
$G_{a,d_1}$ is constructed as drawn in Figure~\ref{fig:v4a5} (a).
In this case, $b_4, d_4, d_6$ are in the same partition.
Then $\widehat{G}_{b_4, b_5}$ has nine edges with a 3-cycle $(b_1 b_2 d_5)$.

\begin{figure}[h]
\includegraphics{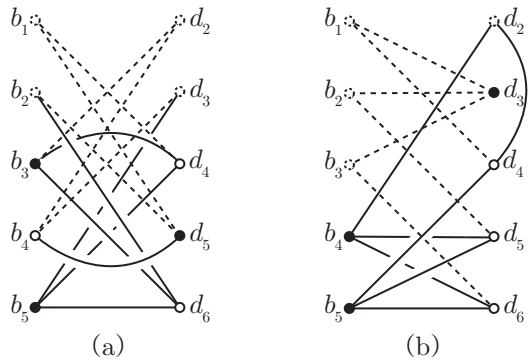}
\caption{Two cases for $|V_4(a)|=5$}
\label{fig:v4a5}
\end{figure}

Suppose that $|V(a) \cap V(d_1)|=3$.
We similarly assume that $d_1$ is adjacent to $b_1, b_2, b_3, d_2$.
If $|V_3(b_1, b_2)| = 0$, $\widehat{G}_{b_1, b_2}$ has nine edges with a degree 4 vertex $b_4$.
If $|V_3(b_1, b_2)| = 2$, $\widehat{G}_{a, d_1}$ has nine edges with a 2-cycle.
Therefore $|V_3(b_1, b_2)| = |V_3(b_1, b_3)| = |V_3(b_2, b_3)| = 1$.
If the three vertices of these three sets are all different, say $d_3, d_4, d_5$
(similar to Figure~\ref{fig:cl21} (a)),
then $\widehat{G}_{a, d_1}$ has nine edges with a 3-cycle $(d_3 d_4 d_5)$.
Therefore we may assume that $d_3$ is adjacent to $b_1, b_2, b_3$ (similar to Figure~\ref{fig:cl21} (b)).
Following the restoring method,
$G_{a,d_1}$ is constructed as drawn in Figure~\ref{fig:v4a5} (b).
In this case, $d_3, d_4, d_5$ are in the same partition.
Then $\widehat{G}_{b_1, b_4}$ has nine edges with a degree 4 vertex $b_5$.

Lastly, suppose that $|V(a) \cap V(d_1)|=4$ which are $b_1, b_2, b_3, b_4$.
As in the previous case, $|V_3(b_i,b_j)|=1$ for any different $i , j = 1,2,3,4$.
But, we can not construct the graph with this property.
Finally we complete the proof.

\end{document}